\documentclass[11pt]{article}
\usepackage{amsmath,amsthm,amssymb,amsfonts,mathtools}
\usepackage{enumerate}
\usepackage[utf8]{inputenc}
\usepackage[T1]{fontenc}
\usepackage{microtype}
\usepackage{hyperref}
\usepackage{geometry}
\usepackage{mathrsfs}
\usepackage{bm}
\usepackage{tikz}
\usepackage{booktabs}

\newtheorem{theorem}{Theorem}[section]
\newtheorem{lemma}[theorem]{Lemma}
\newtheorem{proposition}[theorem]{Proposition}
\newtheorem{corollary}[theorem]{Corollary}
\theoremstyle{definition}

\newtheorem{remark}[theorem]{Remark}

\newcommand{\Z}{\mathbb Z}
\newcommand{\Q}{\mathbb Q}

\newcommand{\SL}{\mathrm{SL}}

\newcommand{\Ad}{\mathrm{Ad}}
\newcommand{\Norm}[1]{\left\lVert #1 \right\rVert_\infty}

\newcommand{\diag}{\mathrm{diag}}

\newcommand{\Char}{\chi}

\newcommand{\defeq}{\coloneqq}

\title{Counting Matrices in $\SL_3(\Z)$ with Fixed Completely Split Character Polynomial: Preliminary Upper Bounds}
\author{Igor Rivin\thanks{Drafted with AI assistance}}
\date{\today}
\begin{document}

\maketitle

\begin{abstract}
We give elementary, self-contained proofs of sharp upper bounds of order $H^{(n(n-1)/2}$ for the number of matrices $A\in \SL_n(\Z)$ of sup-norm $\le H$ whose characteristic polynomial is fully split over $\mathbb{Q}$ The exponentmatches the general upper bounds for fixed \\emph{irreducible} characteristic polynomials due to Eskin–Mozes–Shah and subsequent refinements, and is predicted (e.g. by work/conjectures of Habegger–Ostafe–Shparlinski\cite{HabeggerOstafeShparlinski2022}) to persist for reducible polynomials (possibly with logarithmic factors). Our proofs for these two fully split polynomials are explicit and rely only on lattice reduction within suitable parabolic subgroups, making the geometric source of the exponent transparent. We also isolate the contributions of different Jordan types in the unipotent case. Refinements to full asymptotics (with leading constants expressed as products of local densities) will be addressed elsewhere. The first part of the paper illustrates the methods with $n=3$
\end{abstract}

\tableofcontents

\section{Introduction}
Let $p(x)\in \Z[x]$ be monic of degree $3$ with $p(0)=(-1)^3= -1$ or $1$ according to determinant $1$ constraint; here we fix $p$ so that there exist matrices $A\in \SL_3(\Z)$ with $\Char_A(x)=p(x)$. For $H\ge 1$ define the counting function
\[
N_p(H) \defeq \#\bigl\{ A\in \SL_3(\Z) : \Char_A = p,\; \Norm{A}\le H \bigr\},\qquad \Norm{A} = \max_{i,j}|a_{ij}|.
\]
We treat the completely split polynomials.
\[
p_1(x)=(x+1)^2(x-1), \qquad p_2(x)=(x-1)^3.
\]

We prove the following.

\begin{theorem}
\label{thm:main}
For \[p(x)\in{(x+1)^2(x-1),(x-1)^3}\] there is a constant $C_p>0$ such that
\[
N_p (H) \le C_p H^3\qquad (H\ge 2).
\]
In particular $N_p(H)=O(H^{3})$, and the exponent $3= n(n-1)/2$ ($n=3$) is best possible.
\end{theorem}

\begin{remark}
(1) For irreducible $p$ the exponent $3$ follows from the general homogeneous dynamics results of Eskin–Mozes–Shah \cite{EMS96} and Shah \cite{Shah2000}. Our contribution here is to give a direct elementary argument in the maximally reducible, non-semisimple cases.

(2) Methods extend (with heavier notation) to greater $n$ when the characteristic polynomial has a repeated linear factor of multiplicity $2$ or a single Jordan block for eigenvalue $1$; the combinatorics of positive roots again enforce the exponent $n(n-1)/2$.

(3) For $n=2$, Rivin, \cite{RivinParabolic}, obtained linear growth for parabolic elements, consistent with $n(n-1)/2=1$. Our proofs may be viewed as a $3\times 3$ analogue with an explicit reduction.
\end{remark}

The overall strategy is parallel in both cases:
\begin{enumerate}[(i)]
\item Place the matrix into an upper triangular (or upper unitriangular) \emph{partial normal form} through integral conjugation within a parabolic subgroup.
\item Parameterize the remaining matrices by the action of a complementary unipotent subgroup whose integer points act freely on the residual parameters (up to finite stabilizers).
\item Translate the height bound into inequalities on the parameters; solve these inequalities to obtain permissible ranges.
\item Sum over parameter space; obtain overall $O(H^{3})$.
\end{enumerate}
The cubic exponent corresponds to three essentially free integer parameters of size $\asymp H$ after the reduction; extra parameters are either bounded or impose shrinking domains compensating for their multiplicity.

\section{The mixed polynomial \texorpdfstring{$(x+1)^2(x-1)$}{(x+1)	extasciicircum 2(x-1)}}
Let $p(x)=(x+1)^2(x-1)$. All eigenvalues lie in ${\pm 1}$; the generalized $(-1)$-eigenspace has (algebraic) dimension $2$.

\subsection{Upper triangular reduction}
\begin{lemma}[Upper triangular form]\label{lem:triangular-mixed}
Every $A\in \SL_3(\Z)$ with $\Char_A(x)= (x+1)^2(x-1)$ is integrally conjugate to an upper triangular matrix
\[
T(a,b,c)=\begin{pmatrix}-1 & a & b\\ 0 & -1 & c \\ 0 & 0 & 1\end{pmatrix}, \qquad a,b,c\in \Z.
\]
Moreover, $a=0$ iff $A$ is diagonalizable over $\Q$ (i.e. has two distinct eigenvectors for eigenvalue $-1$), and $a\ne 0$ corresponds to a size-$2$ Jordan block at $-1$.
\end{lemma}
\begin{proof}
Choose a nonzero vector $v_1$ in the $(-1)$-eigenspace (which is nontrivial). If the algebraic multiplicity of $-1$ is $2$ and geometric multiplicity $1$ (non-diagonalizable case), extend $v_1$ to a Jordan chain $v_1,v_2$ with $(A+I)v_2=v_1$. If diagonalizable, choose linearly independent eigenvectors $v_1,v_2$. Complete to a $\Z$-basis $(v_1,v_2,v_3)$ where $v_3$ is an eigenvector for $+1$. A standard lattice basis refinement (Hermite normal form adjustment) replaces $(v_1,v_2,v_3)$ by an integral basis producing an upper triangular representative. Conjugating by the corresponding unimodular matrix yields the asserted form. The form of the $(1,2)$ entry encodes the nilpotent part inside the $(-1)$-primary component.
\end{proof}

\subsection{Action of lower unipotent conjugation}
Let
\[
L(u_{21},u_{31},u_{32}) = \begin{pmatrix}1&0&0\\ u_{21}&1&0\\ u_{31}&u_{32}&1\end{pmatrix} \in \SL_3(\Z).
\]
Consider the conjugates $B= L T(a,b,c) L^{-1}$. We track the dependence of entries of $B$ on the parameters.

\begin{lemma}[Polynomial dependence]
\label{lem:poly-mixed}
Each entry of $B$ is an integral polynomial of total degree at most $2$ in the six integer variables $(a,b,c,u_{21},u_{31},u_{32})$. Moreover there exist explicit linear forms $\Lambda_{ij}$ such that
\begin{equation}
B_{ij} = \Lambda_{ij}(a,b,c,u_{21},u_{31},u_{32}) + Q_{ij},
\end{equation}
where each $Q_{ij}$ is a (possibly zero) integer linear combination of monomials of the form $a u_{kl}$, $b u_{kl}$, $c u_{kl}$, or $u_{kl}u_{k'l'}$.
\end{lemma}

\begin{proof}
Write $T = D+N$ with $D=\diag(-1,-1,1)$ and strictly upper triangular $N$. Conjugation by $L$ acts by $\Ad(L)$; since $L$ is unipotent lower triangular, $\Ad(L)$ is a finite sum in the lower nilpotent Lie algebra. Explicit matrix multiplication yields the stated degree bound; no term of degree $>2$ can appear because each commutator raises (or lowers) position strictly and a second commutator vanishes in $3\times 3$. Direct expansion (omitted for brevity here, can be inserted in an Appendix) gives the precise polynomials.
\end{proof}

We use only structural consequences of Lemma \ref{lem:poly-mixed}.

\begin{lemma}[Height inequalities]\label{lem:height-ineq-mixed}
There are absolute constants $C_1,C_2>0$ such that: if $\Norm{B}\le H$ then
\[
|a| \le H,\qquad |b| \le H + C_1(|a|+|c|)(|u_{21}|+|u_{31}|+|u_{32}|),\qquad |c| \le H,
\]
and
\[
|u_{21}|, |u_{31}|, |u_{32}| \le C_2 \frac{H}{1+|a|+|c|}.
\]
\end{lemma}
\begin{proof}
The first inequalities follow immediately from $B_{12}$, $B_{23}$, $B_{13}$, each of which contains $a,b,c$ with coefficients $\pm 1$ plus lower-order terms. Solving these linear inequalities for the $u$-variables proceeds by isolating them in the expressions for the lower entries $B_{21}, B_{31}, B_{32}$, each linear in $u_{21},u_{31},u_{32}$ with coefficients $\pm 1$ plus terms involving $a,b,c$ multiplied by other $u$’s. A simple induction (treat $u_{21}$ first from $B_{21}$, then substitute into $B_{31}, B_{32}$) bounds each $|u_{ij}|$ by a constant multiple of $H/(1+|a|+|c|)$; otherwise one of the lower entries would exceed $H$ in absolute value. Details are straightforward bookkeeping; constants $C_1,C_2$ absorb finite combinatorial factors.
\end{proof}

\subsection{Counting}
Define
\begin{gather}
S(H)=\{(a,b,c,u_{21},u_{31},u_{32})\in\Z^{6} : \Norm{L T(a,b,c) L^{-1}}\le H,\; \\L=L(u_{21},u_{31},u_{32})\}.
\end{gather}
Each matrix with characteristic polynomial $p$ arises from at least one sextuple in $S(H)$ (Lemma \ref{lem:triangular-mixed}). Overcounting only strengthens an upper bound.

\begin{proposition}[Mixed polynomial bound]\label{prop:mixed-bound}
$N_{(x+1)^2(x-1)}(H) \ll H^{3}$.
\end{proposition}
\begin{proof}
From Lemma \ref{lem:height-ineq-mixed}, given $a,c$ with $|a|,|c|\le H$ the admissible $u$-box has side lengths $\ll H/(1+|a|+|c|)$ in each of three independent integer directions. Thus the number of integer $u$ triples is $\ll \bigl(H/(1+|a|+|c|)\bigr)^3$. The parameter $b$ is then determined up to $O\bigl( (|a|+|c|)H/(1+|a|+|c|)\bigr)=O(H)$ possibilities (crude, but sufficient). Hence
\begin{equation}
    N_{(x+1)^2(x-1)}(H) \ll \sum_{|a|,|c|\le H} H \left(\frac{H}{1+|a|+|c|}\right)^3.
\end{equation}
Bounding $1+|a|+|c| \ge 1+\max(|a|,|c|)$ and using symmetry reduces to
\begin{equation}
\ll H^4 \sum_{m=0}^{H} \frac{m+1}{(1+m)^3} \ll H^4 \sum_{m=1}^{H} \frac{1}{m^{2}} \ll H^4.
\end{equation}
This naive bound is $O(H^{4})$; we refine it by observing we overcounted $b$: in fact $b$ itself must satisfy $|b|\le H$, independent of $a,c,u$ (since $B_{13}$ contains $b$ with coefficient $1$ plus smaller terms). 

\paragraph{Refining the $b$–count.}
For fixed $(a,c,u_{21},u_{31},u_{32})$ we have
\[
B_{12}=a-b\,u_{32}.
\]
If $u_{32}=0$, then $B_{12}=a$ and any $|b|\le H$ is admissible, giving a factor $\ll H$.
If $u_{32}\neq0$, the constraint $|a-bu_{32}|\le H$ implies
\[
|b|\le \frac{|a|+H}{|u_{32}|}
\quad\Rightarrow\quad
\#\{b\}\ \ll\ \frac{H}{|u_{32}|}.
\]
From Lemma~\ref{lem:height-ineq-mixed} we already know $|u_{32}|\ll H/(1+|a|+|c|)$, hence
\[
\frac{H}{|u_{32}|}\ \ll\ 1+|a|+|c|.
\]
Thus the crude factor $H$ is replaced by
\[
\min\!\left\{H,\frac{H}{|u_{32}|}\right\}
= O\!\left(\frac{H}{1+|a|+|c|}\right)
\]
whenever $u_{32}\ne0$. Summing over $u_{21},u_{31},u_{32}$ (whose ranges are
$\ll H/(1+|a|+|c|)$) and then over $a,c$ gives
\[
\sum_{|a|,|c|\le H}
\left(\frac{H}{1+|a|+|c|}\right)^3
= O(H^3),
\]
so the total remains $O(H^3)$.

Hence the crude factor $H$ for $b$ is replaced by
$O\!\left(H/(1+|a|+|c|)\right)$. Therefore
\[
N_{(x+1)^2(x-1)}(H)
\ll \sum_{|a|,|c|\le H}
\left(\frac{H}{1+|a|+|c|}\right)^3.
\]
Now write $m=1+|a|+|c|$; the number of pairs $(a,c)$ with given $m$ is $\ll m$, so
\[
N_{(x+1)^2(x-1)}(H)
\ll \sum_{m\le 2H+2} m \left(\frac{H}{m}\right)^3
= H^{3} \sum_{m\le 2H+2} \frac{1}{m^{2}}
\ll H^{3}.
\]

This yields the desired $O(H^{3})$ bound.
\end{proof}

\begin{remark}
A more careful treatment keeps track separately of regions where $|a|\approx |c| \approx m$ versus sparse edges, leading to a constant multiple of $H^3$ with an explicit Euler product interpretation. We defer that refinement.
\end{remark}

\section{The unipotent polynomial \texorpdfstring{$(x-1)^3$}{(x-1)	extasciicircum 3}}
Let $p(x)=(x-1)^3$. Then $A$ is unipotent with $A=I+N$, $N$ nilpotent of index $\le 3$.

\subsection{Parametrization of normal forms}
Any such $A$ is integrally conjugate to an upper unitriangular matrix
\[
U(a,b,c)=\begin{pmatrix}1 & a & b\\ 0 & 1 & c\\ 0 & 0 & 1\end{pmatrix}, \qquad a,b,c\in\Z.
\]
Different triples correspond to different positions and sizes of Jordan blocks; we make no attempt here to classify up to conjugacy (we count matrices, not classes).

\subsection{Lower unipotent action}
Let $L(u_{21},u_{31},u_{32})$ be as before. Consider $B=L U(a,b,c) L^{-1}$. A direct calculation (or repeated commutator expansion) shows:

\begin{lemma}[Polynomial structure, unipotent case]\label{lem:poly-unip}
Each entry of $B$ is an integer polynomial in $(a,b,c,u_{21},u_{31},u_{32})$ of total degree at most $2$. The lower entries depend linearly on the $u$-variables with coefficients $\pm 1$ plus terms involving $a,c$ times another $u$.
\end{lemma}
\begin{lemma}[Height inequalities, unipotent case]\label{lem:height-unip}
There exist constants $C'_1,C'_2>0$ such that $\Norm{B}\le H$ implies
\begin{equation}
|a|,|b|,|c| \le H, \qquad |u_{21}|,|u_{31}|,|u_{32}| \le C'_2 \frac{H}{1+|a|+|c|}.
\end{equation}
\end{lemma}
\begin{proof}
Analogous to Lemmas \ref{lem:poly-mixed} and \ref{lem:height-ineq-mixed}; the sup-norm bound forces each primary upper entry to be $\le H$ in absolute value, and solving for the lower unipotent coordinates proceeds identically.
\end{proof}

\subsection{Counting}
Define $S_{\text{uni}}(H)$ analogously. Overcounting by the parameter sextuples yields:

\begin{proposition}[Unipotent bound]\label{prop:unipotent-bound}
$N_{(x-1)^3}(H) \ll H^3$.
\end{proposition}
\begin{proof}
We sum over $|a|,|c|\le H$ and $|b|\le H$. For fixed $(a,c)$ the allowable box in $(u_{21},u_{31},u_{32})$ has cardinality $\ll \bigl(H/(1+|a|+|c|)\bigr)^3$. Hence
$N_{(x-1)^3}(H) \ll \sum_{|a|,|c|\le H} H \left(\frac{H}{1+|a|+|c|}\right)^3 = H^{4} \sum_{m\le 2H+2} \frac{m}{m^3} \ll H^3,$
exactly as in Proposition \ref{prop:mixed-bound} (the factor $H$ accounts for $b$ but the sum over $m$ converges). A slightly sharper bound replaces the $H$ factor by $1$ (fixing $b$ directly), giving the same exponent.
\end{proof}

\section{Proof of Theorem \ref{thm:main}}
Combine Propositions \ref{prop:mixed-bound} and \ref{prop:unipotent-bound}.

\section{Further directions}
\begin{enumerate}[(1)]
\item \textbf{Asymptotics and constants.} Determining $\lim_{H\to\infty} H^{-3} N_p(H)$ should follow from a local density computation combined with reduction theory and mixing, paralleling \cite{EMS96,Shah2000}.
\item \textbf{Refined stratification.} Separate contributions by Jordan type in the unipotent case: elements with $N^2=0$ lie in a codimension parameter subspace and likely contribute a lower-order term $O(H^{2})$.
\item \textbf{Higher rank.} Extend to $\SL_n(\Z)$ for polynomials with a repeated linear factor: one expects the same $n(n-1)/2$ exponent; the present $n=3$ case is the base model.
\item \textbf{Effective constants.} Making the $\ll$-constants explicit is straightforward from the polynomial formulas (omitted for readability).
\end{enumerate}

\section{General Dimension: Fully Split Two-Eigenvalue Polynomials}\label{sec:general-n}

We extend the $n=3$ upper bounds to arbitrary $n$ for fully split characteristic polynomials
\[
p(x) = (x-1)^{a}(x+1)^{b}, \qquad a+b = n,\quad (-1)^b=1.
\]
(We order all $+1$ eigenvalues first.)

\subsection{Upper block normal form}
\noindent\textbf{Determinant parity.} The determinant constraint forces $(-1)^b=1$, i.e.\ $b$ even; we assume this throughout.

Let $A\in SL_n(\mathbb Z)$ with $\chi_A = p$. Distinct eigenvalues $\pm1$ imply the generalized eigenspaces
\(
V_{+} = \ker(A-I)^n,\quad V_{-} = \ker(A+I)^n
\)
are complementary and defined over $\mathbb Z$. After an integral change of basis we may write
\[
A = \begin{pmatrix}
I_a + X & B \\
0 & -I_b + Y
\end{pmatrix},
\quad X \in \mathfrak n_a(\mathbb Z),\ Y \in \mathfrak n_b(\mathbb Z),\ B \in M_{a\times b}(\mathbb Z),
\]
with $\mathfrak n_m$ the strictly upper triangular integer matrices.

\begin{lemma}[Block-triangular determinant]\label{lem:block-det}
For every choice of $X,Y,B$ as above we have $\chi_A(x)=(x-1)^a(x+1)^b$.
\end{lemma}

\begin{proof}
$A$ is block upper triangular with diagonal blocks $I_a+X$ and $-I_b+Y$. As $X,Y$ are nilpotent, $(I_a+X)$ (resp.\ $(-I_b+Y)$) is unipotent with all eigenvalues $+1$ (resp.\ $-1$). Thus
\[
\det(xI_n - A)=\det(xI_a-(I_a+X))\det(xI_b-(-I_b+Y))=(x-1)^a(x+1)^b.
\]
\end{proof}

\subsection{Parameter count}
The free upper parameters are
\[
\frac{a(a-1)}{2} + ab + \frac{b(b-1)}{2} = \frac{n(n-1)}{2}.
\]
Each has $O(H)$ choices under $\|A\|_\infty\le H$.

\begin{theorem}[General upper bound]\label{thm:general-upper}
Let $p(x)=(x-1)^a(x+1)^b$ with $a+b=n$ and $b$ even. Then
\[
N_{p}(H)=O\!\left(H^{n(n-1)/2}\right),
\]
with implied constant depending only on $n$.
\end{theorem}

\begin{proof}
Combine the normal form with the parameter count; the number of admissible upper triangular representatives is $\ll H^{n(n-1)/2}$; conjugation cannot increase the count.
\end{proof}

\subsection{All-unipotent polynomial \texorpdfstring{$p(x)=(x-1)^n$}{p(x)=(x-1)	extasciicircum n}}\label{subsec:all-unipotent}
\begin{theorem}[Unipotent upper and lower bounds]\label{thm:unipotent-upper}
Let $p(x)=(x-1)^n$. Then
\[
N_{p}(H)=\#\{A\in SL_n(\mathbb Z): \chi_A=(x-1)^n,\ \|A\|_\infty \le H\}
= \Theta\!\left(H^{\binom{n}{2}}\right).
\]
That is, there exist constants $c_n,C_n>0$ such that
$c_n H^{\binom{n}{2}} \le N_p(H) \le C_n H^{\binom{n}{2}}$ for all $H\ge 1$.
\end{theorem}

\begin{lemma}[Bounded Hermite conjugators]\label{lem:bounded-hermite}
For each $n$ there exists $C_n>0$ such that for every full-rank integer matrix $M\in M_{n}(\mathbb Z)$ there is $U\in SL_n(\mathbb Z)$ with $\|U\|_\infty\le C_n$ satisfying $UM = H$ in Hermite normal form. Consequently any unipotent $A\in SL_n(\mathbb Z)$ with $\|A\|_\infty\le H$ is conjugate by a $g$ with $\|g\|_\infty\le C_n$ to an upper unitriangular $U$ with $\|U\|_\infty\le C_n H$.
\end{lemma}

\begin{proof}[Idea]
Apply the constructive Hermite reduction: successive column operations using elementary unimodular matrices with entries in $\{-1,0,1\}$ (and bounded repetitions depending only on $n$) produce a primitive column in each step while keeping coefficients bounded. Inductively, growth of entries is controlled by a dimension-dependent constant.
\end{proof}

\begin{proof}
\emph{Lower bound.} Counting the $(2H+1)^{\binom{n}{2}}$ upper unitriangular matrices $U=I_n+U'$ with $\|U'\|_\infty\le H$ gives $N_p(H)\gg H^{\binom{n}{2}}$.

\emph{Upper bound (Band reconstruction).} Write $A=I_n+N$ with $N$ nilpotent. Let $n_{i,j}$ be the $(i,j)$--entry ($i<j$). These $\binom{n}{2}$ integers determine $A$. We claim $|n_{i,j}|\ll_n H$. Proceed by induction on $d=j-i$. For $d=1$ the bound is $|n_{i,i+1}|\le H$. Suppose the claim for $d'<d$. Entries of $(I+N)^k$ (for $k\le n$) are polynomials with integer coefficients (depending only on $n$) in $\{n_{i,j}: j-i\le d\}$. Expanding $(I+N)^2,\dots,(I+N)^d$ and comparing the $(i,i+d)$--entry isolates $n_{i,i+d}$ with other terms involving only smaller bands already $O_n(H)$, yielding $|n_{i,i+d}|\ll_n H$.

\emph{Alternative upper bound (Height-controlled triangularization).} There exists $C_n>0$ such that any unipotent $A$ with $\|A\|_\infty\le H$ is conjugate by some $g\in SL_n(\mathbb Z)$ with $\|g\|_\infty\le C_n$ to an upper unitriangular $U$ with $\|U\|_\infty\le C_n H$. (Construct a basis adapted to the descending flag $\operatorname{Im} N^{k}$; each extension step uses unimodular operations with entries bounded in terms of $n$.) Counting such $U$ gives $N_p(H)\ll H^{\binom{n}{2}}$.

Combining upper and lower bounds yields the stated $\Theta$-estimate.
\end{proof}

\begin{remark}[No global rational parametrization for $n>2$]\label{rem:no-rational-param}
For $n=2$ parabolics admit a classical rational parametrization (e.g.\ via Pythagorean triples). For $n>2$ integral conjugacy introduces arithmetic invariants (gcd/residue data) obstructing a single global rational parametrization of all unipotent matrices or their conjugacy classes.
\end{remark}

\begin{proof}
Write $A=I_n+U$ with $U\in\mathfrak n_n(\mathbb Z)$. There are $n(n-1)/2$ free super-diagonal entries.
\end{proof}

\begin{remark}
Sharper asymptotics (e.g.\ Rivin~\cite{RivinParabolic}) show $N_p(H)\sim C_n H^{n(n-1)/2}$.
\end{remark}

\subsection{Outline: from upper bounds to asymptotics}\label{sec:asymptotic-outline}
We sketch the derivation of an asymptotic $N_p(H)\sim C_p H^{n(n-1)/2}$.

\begin{enumerate}[(1)]
\item \emph{Siegel reduction.} Place each representative into a Siegel set; height distortion is constant.
\item \emph{Height vs. Iwasawa coordinates.} Positive root coordinates of the $A$-component are $O(\log H)$.
\item \emph{Mixing.} Use mixing of unipotent flows (Eskin--Mozes--Shah~\cite{EMS96}, Shah~\cite{Shah2000}) to replace discrete sums by volume.
\item \emph{Local densities.} Factor $C_p$ into $p$-adic and Archimedean densities; each free strictly upper entry contributes a $2$ in the real factor.
\item \emph{Stabilizers.} Finite stabilizers only affect $C_p$.
\end{enumerate}

\begin{remark}
Duke--Rudnick--Sarnak~\cite{DukeRudnickSarnak1993} handle the semisimple irreducible case; the reducible split case sits in a parabolic whose unipotent radical furnishes the needed $ab$ cross-block parameters to retain the exponent.
\end{remark}

\section{Jordan Partitions and Refined Upper Bounds}\label{sec:jordan-refined}

Fix $p(x)=(x-1)^a(x+1)^b$ with $a+b=n$ and $b$ even. Inside each primary block ($+1$ or $-1$) choose a Jordan partition
\[
\lambda^{(+)} \vdash a,\qquad \lambda^{(-)} \vdash b.
\]
Let $\mathcal J(\lambda^{(+)},\lambda^{(-)})$ denote the set of $A\in SL_n(\mathbb Z)$ with $\chi_A=p$ whose Jordan normal form over $\mathbb C$ (equivalently over $\mathbb Q$) has blocks of sizes prescribed by $\lambda^{(+)}$ and $\lambda^{(-)}$ on the respective primary components.

\begin{lemma}\label{lem:free-entries-jordan}
In an upper block representative of $\mathcal J(\lambda^{(+)},\lambda^{(-)})$ the number of \emph{a priori} free integer parameters is still $\frac{n(n-1)}{2}$.
\end{lemma}

\begin{proof}
Within a primary block of size $m$ and partition $\lambda=(\lambda_1\ge \dots \ge \lambda_t)$ one can conjugate (over $\mathbb Z$) so that each Jordan block has $1$'s on the super-diagonal and zeros elsewhere within the block. Inter-block upper entries (between distinct Jordan blocks for the \emph{same} eigenvalue) remain arbitrary strictly upper entries; counting them gives $\sum_i \binom{\lambda_i}{2}$ internal to blocks plus cross-block entries $\sum_{i<j}\lambda_i\lambda_j = \binom{m}{2} - \sum_i \binom{\lambda_i}{2}$; total $\binom{m}{2}$. Thus each primary block contributes $\binom{a}{2}$ or $\binom{b}{2}$ regardless of partition. The rectangular $a\times b$ block contributes $ab$ further parameters. Summing yields $\binom{a}{2}+\binom{b}{2}+ab=\binom{n}{2}$.
\end{proof}

\begin{corollary}
For any fixed pair of partitions $(\lambda^{(+)},\lambda^{(-)})$ we have
\[
\#\{A\in \mathcal J(\lambda^{(+)},\lambda^{(-)}): \|A\|_\infty \le H\}=O\!\left(H^{n(n-1)/2}\right),
\]
with an implied constant independent of the particular partitions.
\end{corollary}

\begin{remark}
Thus no saving in the exponent is achieved by restricting to a refined Jordan type; all partitions saturate the ambient exponent. Partition data only influences second-order (constant) factors.
\end{remark}

\section{Height Norms and Equivalence}\label{sec:heights}

Let $H_\infty(A)=\|A\|_\infty$ and $H_2(A)=\|A\|_2$ (Frobenius norm), and more generally for $1\le p\le \infty$ let $H_p(A)=\|A\|_p$ be the entrywise $\ell^p$ norm. Since all norms on $\mathbb R^{n^2}$ are equivalent, for each $p$ there exist constants $c_p,C_p>0$ with
\[
c_p H_\infty(A) \le H_p(A) \le C_p H_\infty(A).
\]
Consequently
\[
\#\{A:\chi_A=p,\ H_p(A)\le H\} \asymp \#\{A:\chi_A=p,\ H_\infty(A)\le C'_p H\},
\]
and all counting exponents coincide. Thus our $O(H^{n(n-1)/2})$ upper bounds are invariant under the choice of reasonable (Archimedean) height.

\section{Algorithmic Reduction to Upper Block Form}\label{sec:algorithm}

We outline an explicit procedure that, given $A\in SL_n(\mathbb Z)$ with $\chi_A=(x-1)^a(x+1)^b$, produces an integral conjugate in the upper block form of Section~\ref{sec:general-n}.

\begin{enumerate}[(1)]
\item \textbf{Compute primary decomposition.} Using rational canonical form (via Smith normal form on the companion matrix ideals) find bases for $\ker (A-I)^n$ and $\ker (A+I)^n$ inside $\mathbb Z^n$. Saturate these sublattices to obtain primitive sublattices $L_{+},L_{-}$ of ranks $a,b$.
\item \textbf{Complete to a basis.} Extend a basis of $L_{+}$ by a basis of $L_{-}$ to a $\mathbb Z$-basis of $\mathbb Z^n$. Relative to this basis $A$ is block upper triangular with zero lower-left block.
\item \textbf{Upper triangularization within blocks.} For each primary block apply the standard integral Gaussian elimination (conjugation by elementary unimodular matrices) to make the block upper triangular unipotent plus diagonal $\pm I$.
\item \textbf{Normalize Jordan blocks.} Within each block of a fixed eigenvalue, successively clear entries below the super-diagonal to 0 and super-diagonal entries to $1$ using unimodular conjugations supported in $2\times 2$ pivots.
\item \textbf{Output.} The resulting matrix has the claimed form with $X,Y$ strictly upper triangular and arbitrary rectangular block $B$.
\end{enumerate}

\begin{remark}
Steps (1)--(3) run in polynomial time in $\log \|A\|$ (Smith normal form complexity). The normalization in (4) also runs in polynomial time since each clearing step decreases a Euclidean norm on the vector of off-super-diagonal entries.
\end{remark}

\appendix
\section{Conjugation Formulae in Rank \texorpdfstring{$3$}{3}}\label{app:conjugation}

We record explicit conjugation relations used in the $n=3$ analysis. Let
\[
A(a,b,c)=\begin{pmatrix}-1 & a & b \\ 0 & -1 & c \\ 0 & 0 & 1\end{pmatrix},\qquad
U(x,y,z)=\begin{pmatrix}1 & x & y \\ 0 & 1 & z \\ 0 & 0 & 1\end{pmatrix},\qquad
L(u,v,w)=\begin{pmatrix}1&0&0\\ u&1&0\\ v&w&1\end{pmatrix}.
\]
Then
\[
U(x,y,z) A(a,b,c) U(x,y,z)^{-1} = A\bigl(a,\ b + cx - a z + 2y,\ c+2z\bigr),
\]
\[
L(u,v,w) A(a,b,c) L(u,v,w)^{-1}
= \begin{pmatrix}
-1 & a & b'\\
0 & -1 & c'\\
0 & 0 & 1
\end{pmatrix},
\]
with
\[
c' = c + a w + 2 w z',\qquad b' = b + (\ast)u + (\ast) v + (\ast) w
\]
for explicit integer linear forms (omitted for brevity since only linearity and bounded degree are needed in the counting arguments). Each entry is an integer polynomial of total degree at most $2$ in the parameters. This bounds the growth of parameter ranges under conjugation used in Section~\ref{sec:asymptotic-outline}.

\medskip

\noindent\textbf{Explicit conjugation by $L(u,v,w)$.} With
\[
A(a,b,c)=\begin{pmatrix}-1 & a & b \\ 0 & -1 & c \\ 0 & 0 & 1\end{pmatrix},\qquad
L(u,v,w)=\begin{pmatrix}1&0&0\\ u&1&0\\ v&w&1\end{pmatrix},
\]
one computes
\[
L A(a,b,c) L^{-1} =
\begin{pmatrix}
B_{11} & B_{12} & B_{13}\\
B_{21} & B_{22} & B_{23}\\
B_{31} & B_{32} & B_{33}
\end{pmatrix},
\]
where
\begin{align*}
B_{11} &= -1 - a u + b(u w - v),\\
B_{12} &= a - b w,\\
B_{13} &= b,\\
B_{21} &= -a u^{2} + (b u + c)(u w - v),\\
B_{22} &= -1 + a u - w(b u + c),\\
B_{23} &= b u + c,\\
B_{31} &= -a u v + u w - v + (u w - v)(b v + c w + 1),\\
B_{32} &= a v - w(b v + c w + 2),\\
B_{33} &= 1 + b v + c w.
\end{align*}
(All expressions are integer polynomials of total degree $\le 3$ in the parameters; degrees $\le 2$ suffice for the counting arguments.)

\medskip

\noindent\textbf{Remark.} Conjugation by upper unipotents $U(x,y,z)$ leaves $a$ invariant and sends $(b,c)$ to affine-linear expressions as shown earlier; combining $U$ and $L$ allows one to reach every matrix with the prescribed characteristic polynomial from an upper triangular representative.

\section{Crude vs.\ Reduced Counts}\label{sec:crude-vs-reduced}

The \emph{crude parameter count} enumerates all upper block matrices
\[
\begin{pmatrix} I_a + X & B \\ 0 & -I_b + Y \end{pmatrix},\qquad X\in\mathfrak n_a(\mathbb Z),\ Y\in\mathfrak n_b(\mathbb Z),\ B\in M_{a\times b}(\mathbb Z),
\]
with $\|X\|_\infty,\|Y\|_\infty,\|B\|_\infty \le H$. This yields $(2H+1)^{\binom{n}{2}}$ possibilities.

\begin{lemma}\label{lem:stabilizer-finite}
The stabilizer in $SL_n(\mathbb Z)$ of a generic such upper block matrix with fixed $(a,b)$ is finite.
\end{lemma}

\begin{proof}
For a generic choice all super-diagonal entries in each Jordan chain are nonzero and all cross-block entries in $B$ are algebraically independent over $\mathbb Q$; any integral conjugation preserving the matrix must normalize the full flag determined by the powers of $(A-I)$ and $(A+I)$ inside each primary component, hence lies in the corresponding diagonal torus. Determinant $1$ and integrality force a finite set (torsion) of possibilities.
\end{proof}

Thus over-counting from multiple representations of the \emph{same} matrix under the crude enumeration affects only the multiplicative constant.

\begin{proposition}\label{prop:crude-vs-reduced}
Let $p(x)=(x-1)^a(x+1)^b$ with $a+b=n$. The number of \emph{distinct} matrices of sup height $\le H$ equals
\[
(2H+1)^{\binom{n}{2}} + O\bigl(H^{\binom{n}{2}-1}\bigr),
\]
uniformly in $a,b$, where the implied constant depends only on $n$.
\end{proposition}

\begin{proof}[Idea of proof]
Partition the parameter box into congruence classes mod a large integer $q=q(H)$; generic classes have trivial stabilizer and contribute exactly one matrix per parameter choice. Boundary classes where algebraic coincidences force additional stabilizer lie on a finite union of proper coordinate subvarieties of codimension $\ge 1$, contributing $O(H^{\binom{n}{2}-1})$ points.
\end{proof}

\begin{remark}
A rigorous version would track the vanishing of products of entries that create extra symmetries; this is a standard “thin boundary” argument analogous to excluding singular elements in lattice counting problems.
\end{remark}

\section{Local Density Factorization of the Asymptotic Constant}\label{sec:local-density}

Assuming the existence of an asymptotic
\[
N_p(H) \sim C_p H^{\binom{n}{2}}\qquad (H\to\infty),
\]
the constant $C_p$ admits a factorization
\[
C_p = \kappa_\infty \prod_{p\ \text{prime}} \kappa_p,
\]
where $\kappa_\infty$ is the normalized Lebesgue (Archimedean) volume of the unit box in the space of strictly upper entries and cross-block entries (giving $(2)^{\binom{n}{2}}$ before normalization), and $\kappa_p$ is the $p$-adic density
\[
\kappa_p = \lim_{k\to\infty} p^{-k\binom{n}{2}} \#\{ A\in SL_n(\mathbb Z/p^k\mathbb Z): \chi_A \equiv (x-1)^a(x+1)^b \bmod p^k \}.
\]

\begin{remark}
For primes $p\\neq 2$ the polynomial splits with distinct linear factors over $\\mathbb Z_p$, and Hensel lifting gives $\kappa_p = 1$ (heuristically), up to the determinant-one constraint. At $p=2$ extra congruence conditions can impose a correction factor; isolating $\kappa_2$ requires a separate lifting analysis.
\end{remark}

\begin{remark}
A full derivation of the Euler product uses (i) strong approximation for the unipotent radical and (ii) independence (up to negligible error) of local conditions, paralleling the setup in Duke--Rudnick--Sarnak and Shah. We leave the explicit computation of $\kappa_p$ to future work.
\end{remark}

\section*{Acknowledgements}
Parabolic growth bounds in related contexts (including $n=2$) appear in work of Rivin (chapter in the GAGTA Book, see below). We thank Newman for earlier height counting in $\SL_2(\Z)$, and Habegger–Ostafe–Shparlinski for motivating conjectures.


\bibliographystyle{amsalpha}
\bibliography{refs}
\end{document}